\documentclass[10pt]{article}

\usepackage{hyperref}
\usepackage{xcolor}
\usepackage{amssymb}
\usepackage{amsfonts}
\usepackage{amsthm}
\usepackage{amsmath}
\usepackage{float}
\usepackage{bigints}
\usepackage{fullpage}
\usepackage{mathtools}
\usepackage[utf8]{inputenc} 
\usepackage{cancel}
\usepackage{verbatim}
\usepackage{graphicx}

\newtheorem{theorem}{Theorem}

\newtheorem{lemma}[theorem]{Lemma}

\newtheorem{proposition}[theorem]{Proposition}
\newtheorem{claim}[theorem]{Claim}
\newtheorem{conjecture}[theorem]{Conjecture}

\newtheorem{corollary}[theorem]{Corollary}
\newtheorem{remark}[theorem]{Remark}

\newcommand{\x}{\mathbf{x}}
\newcommand{\p}{\mathbf{p}}
\newcommand{\q}{\mathbf{q}}
\newcommand{\y}{\mathbf{y}}
\newcommand{\z}{\mathbf{z}}
\newcommand{\uu}{\mathbf{u}}
\newcommand{\oo}{\mathbf{o}}

\newcommand{\ccc}{\mathbf{c}}

\newcommand{\B}{\mathbf{B}}

\newcommand{\Ee}{\mathbb{E}}

\newcommand{\Ed}{\Ee^d}

\newcommand{\noshow}[1]{}
\newcommand{\Sedm}{{\mathbb S}^{d-1}}

\newcommand{\Ze}{{\mathbb Z}}

\title{On contact numbers of locally separable unit sphere packings \footnote{\hspace*{.35cm} 2010 Mathematics Subject Classification: 52C17, 52C05.}}

\author{K\'{a}roly Bezdek\thanks{Partially supported by a Natural Sciences and 
Engineering Research Council of Canada Discovery Grant.}}

\date{}

\begin{document}

\maketitle

\begin{abstract}
\noindent  

The contact number of a packing of finitely many balls in Euclidean $d$-space is the number of touching pairs of balls in the packing. A prominent subfamily of sphere packings is formed by the so-called totally separable sphere packings: here, a packing of balls in Euclidean $d$-space is called totally separable if any two balls can be separated by a hyperplane such that it is disjoint from the interior of each ball in the packing. Bezdek, Szalkai and Szalkai (Discrete Math. 339(2): 668-676, 2016) upper bounded the contact numbers of totally separable packings of $n$ unit balls in Euclidean $d$-space in terms of $n$ and $d$. In this paper we improve their upper bound and extend that new upper bound to the so-called locally separable packings of unit balls. We call a packing of unit balls a locally separable packing if each unit ball of the packing together with the unit balls that are tangent to it form a totally separable packing. In the plane, we prove a crystallization result by characterizing all locally separable packings of $n$ unit disks having maximum contact number.
\end{abstract}

\section{Introduction}\label{sec:intro}

Let $\Ee^d$ denote the $d$-dimensional Euclidean space, with inner product $\langle\cdot ,\cdot\rangle$ and norm $\|\cdot\|$. Its unit sphere centered at the origin $\oo$ is $\Sedm:= \{\x\in\Ee^d\ |\ \|\x\|= 1\}$. The closed Euclidean ball of radius $r$ centered at $\p\in\Ed$ is denoted by $\B^d[\p,r]:=\{\q\in\Ed\ |\  \|\p-\q\|\leq r\}$.  Lebesgue measure on $\Ee^d$ is denoted by ${\rm vol}_d(\cdot)$. Let $\mathcal {P}:=\{\B^d[\ccc_i,1]|1\leq i\leq n\}$ be a packing of $n>1$ unit balls in $\Ee^d$ (i.e., let $\|\ccc_i-\ccc_j\|\geq 2$ for all $1\leq i<j\leq n$). Recall that the {\it contact graph} $G_c(\mathcal {P})$ of $\mathcal {P}$ is the simple graph whose vertices correspond to the packing elements $\B^d[\ccc_i,1],1\leq i\leq n$, and whose two vertices corresponding to say, $\B^d[\ccc_i,1]$ and $\B^d[\ccc_j,1]$ are connected by an edge if and only if  $\B^d[\ccc_i,1]$ and $\B^d[\ccc_j,1]$ are tangent to each other (i.e., $\|\ccc_i-\ccc_j\|=2$). The number of edges of $G_c(\mathcal {P})$ is called the {\it contact number} of  $\mathcal {P}$ and we denote it by $c(\mathcal {P})$. The {\it contact number problem} consists  of maximizing $c(\mathcal {P})$ for packings $\mathcal {P}$ of $n$ unit balls in $\Ee^d$. The answer to this challenging problem is known for all $n>1$ only in $\Ee^2$. Namely, it was proved in \cite{Har} (see also \cite{HeRa}) that the largest contact number of packings of $n>1$ unit disks in $\Ee^2$ is equal to $\lfloor 3n-\sqrt{12n-3}\rfloor$, where $\lfloor\cdot\rfloor$ denotes the lower integer part of the given real. For a comprehensive survey on the contact number problem see \cite{BeKh}. The closely related question on maximizing the degree of a vertex of $G_c(\mathcal {P})$ is an even older and notoriously difficult problem, which has been widely investigated under {\it kissing numbers}. Here, the kissing number $\tau_d$ of a $d$-dimensional unit ball is the largest number of non-overlapping unit balls that can simultaneously touch another unit ball in $\Ee^d$. The value of $\tau_d$ is only known for $d=1,2,3,4,8, 24$. For an extensive survey, on kissing numbers we refer the interested reader to \cite{BoDoMu}. The current best upper bounds for contact numbers of unit ball packings in $\Ee^d$, $d\geq 3$ are the following: It is proved in \cite{Be02} that if $\mathcal {P}$ is a packing of $n>1$ unit balls in $\Ee^d$, $d\geq 4$, then $c(\mathcal {P})\leq\left\lfloor\frac{1}{2}\tau_dn-\left(\frac{1}{2}\right)^d\delta_d^{-\frac{d-1}{d}}n^{\frac{d-1}{d}}\right\rfloor$, where $\delta_d$ stands for the supremum of the upper densities of unit ball packings in $\Ee^d$. On the other hand, it is proved in \cite{BeRe} that if $\mathcal {P}$ is a packing of $n>1$ unit balls in $\Ee^3$, then $c(\mathcal {P})\leq \left\lfloor6n-0.926n^{\frac{2}{3}}\right\rfloor$.

A prominent subfamily of sphere packings is formed by the so-called {\it totally separable (TS)} sphere packings: here, a packing  $\mathcal {P}$ of balls is called a {\it TS-packing} in $\Ee^d$ if any two balls of $\mathcal {P}$ can be separated by a hyperplane of $\Ee^d$ such that it is disjoint from the interior of each ball in $\mathcal {P}$. This notion was introduced by Fejes T\'oth and Fejes T\'oth \cite{FeFe} and has attracted significant attention. The recent paper \cite{BeSzSz} investigates contact graphs of TS-packings of unit balls by proving the following theorems. On the one hand, the largest number of non-overlapping unit balls that can simultaneously touch another unit ball forming a TS-packing in $\Ee^d$ is equal to $2d$ for all $d\geq 2$. On the other hand, the largest contact number of TS-packings of $n>1$ unit disks in $\Ee^2$ is equal to $\lfloor 2n-2\sqrt{n}\rfloor$. Moreover, if $\mathcal {P}$ is an arbitrary TS-packing of $n>1$ unit balls in $\Ee^3$, then $c(\mathcal {P})\leq\big\lfloor 3n - 1.346n^{\frac{2}{3}}\big\rfloor$. Finally, if $\mathcal {P}$ is any TS-packing of $n>1$ unit balls in $\Ee^d, d\geq 4$, then 
\begin{equation}\label{BeSzSz-Main}
c(\mathcal {P})\leq \bigg\lfloor dn-\frac{1}{2}d^{-\frac{d-1}{2}}n^{\frac{d-1}{d}}\bigg\rfloor .
\end{equation}

One of the main goals of this note is to improve the upper bound of (\ref{BeSzSz-Main}) and to extend it to the so-called {\it locally separable (LS)} unit ball packings. Here, we call a packing $\mathcal {P}$ of balls an {\it LS-packing} in $\Ee^d$ if each ball of $\mathcal {P}$ together with the balls of $\mathcal {P}$ that are tangent to it form a TS-packing in $\Ee^d$. Clearly, any TS-packing is also an LS-packing, but not necessarily the other way around (Figure~\ref{Figure0.pdf}). 

\begin{figure}[ht]
\begin{center}
\includegraphics[width=0.6\textwidth]{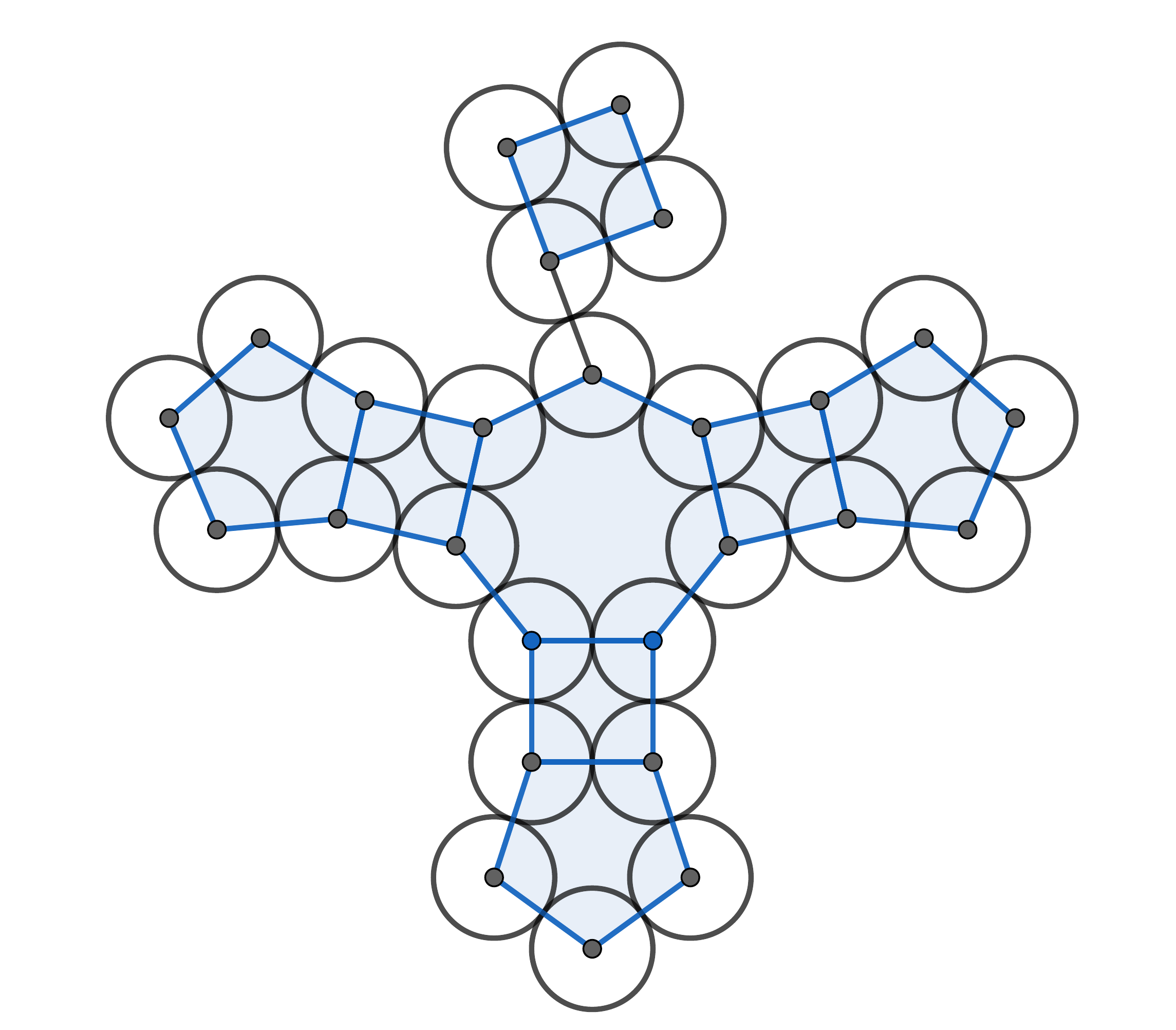}
\caption{An LS-packing of unit disks which is not a TS-packing.}
\label{Figure0.pdf} 
\end{center}
\end{figure}

In fact, the family of LS-packings of unit balls is a rather large and complex family including also the family of the so-called {\it $\rho$-separable} packings of unit balls introduced and investigated in \cite{BeLa}. Let $\delta_d^{LS}$ (resp., $\delta_d^{TS}$) denote the supremum of the upper densities of LS-packings (rep., TS-packings) of unit balls in $\Ee^d$. It is easy to see that $\delta_d^{LS}=\delta_d$. On the other hand, we know only the following exact values of $\delta_d^{TS}$:  $\delta_2^{TS}=\frac{\pi}{4}$ (\cite{FeFe}) and $\delta_3^{TS}=\frac{\pi}{6}$ (\cite{Ke}). Here, it is natural to expect that Kert\'esz's theorem stated and proved in \cite{Ke} extends to higher dimensions, which leads us to the following difficult looking problem.

\begin{conjecture}\label{Bezdek-1}
If a $d$-dimensional cube ${\bf Q}$ contains a TS-packing of $N$ unit balls in $\Ee^d$, $d\geq 4$, then ${\rm vol}_d({\bf Q})\geq N2^d$ implying that $\delta_d^{TS}=2^{-d}\omega_d$, where $\omega_d:={\rm vol}_d\left( \B^d[\oo,1]\right)$ and $2^{-d}\omega_d$ is the density of the TS-packing of unit diameter balls centered at the points of the integer lattice $\Ze^d$ in $\Ee^d$.
\end{conjecture}

If Conjecture~\ref{Bezdek-1} holds, then combining it with $\delta_d\geq\Omega(d2^{-d})$ (\cite{JeJoPe}) and $\lim_{d\to\infty}\omega_d=0$, yields that $\delta_d^{TS}<\delta_d(=\delta_d^{LS})$ holds for any sufficiently large $d$. In other words, Conjecture~\ref{Bezdek-1} implies that LS-packings of unit balls behave very differently from TS-packings of unit balls in terms of density. Next, let $c_{TS}(n, d)$ (resp., $c_{LS}(n, d)$) denote the largest contact number of TS-packings (resp., LS-packings) of $n$ unit balls in $\Ee^d$. We expect that TS-packings and LS-packings behave in a similar way from the point of view of the contact number problem, which leads us to the following natural
question.

\begin{conjecture}\label{Bezdek-2}
$c_{TS}(n, d)=c_{LS}(n, d)$ for all $n>1$ and $d>1$.
\end{conjecture}

Corollary~\ref{2D-corollary-1} proves Conjecture~\ref{Bezdek-2} for all $n>1$ and $d=2$. Furthermore, in this note we make the first steps towards investigating Conjecture~\ref{Bezdek-2} in higher dimensions by extending and improving earlier contact number estimates on TS-packings to LS-packings in Theorem~\ref{Main-theorem}. The details are as follows.

In order to state our first main result in a precise form we need the following notation. Let $\mathcal {P}:=\{\B^d[\ccc_i,1]| i\in I\}$ be an arbitrary (finite or infinite) packing of unit balls in $\Ee^d$, $d\geq 3$ and let $\mathbf{V}_i:=\{\x\in \Ee^d | \ \|\x-\ccc_i\|\leq \|\x-\ccc_j\|\ {\rm for \ all}\ j\neq i, j\in I\}$ denote the Voronoi cell assigned to $\B^d[\ccc_i,1]$ for $i\in I$. Recall (\cite{Rog}) that the Voronoi cells $\{\mathbf{V}_i | i\in I\}$ form a face-to-face tiling of $\Ee^d$. Then let the largest density of the unit ball $\B^d[\ccc_i,1]$ in its truncated Voronoi cell $\mathbf{V}_i\cap \B^d[\ccc_i,\sqrt{d}]$ be denoted by $\hat{\delta}_d$, i.e., let
$\hat{\delta}_d:=\sup_{\mathcal {P}}\left( \sup_{i\in I}  \frac{\omega_d}{{\rm vol}_d\left(\mathbf{V}_i\cap \B^d[\ccc_i,\sqrt{d}]\right)} \right)$,
where $\mathcal {P}$ runs through all possible unit ball packings of $\Ee^d$. We are now ready to state our first main result.

\begin{theorem}\label{Main-theorem} Let $\mathcal {P}$ be an arbitrary LS-packing of $n>1$ unit balls in $\Ee^d$, $d\geq 3$. Then
\begin{equation}\label{Main-1}
c(\mathcal {P})\leq \left\lfloor dn-\left(d^{-\frac{d-3}{2}}\hat{\delta}_d^{-\frac{d-1}{d}}\right)n^{\frac{d-1}{d}} \right\rfloor .
\end{equation}

\end{theorem}

\begin{remark}\label{strengthening}
We note that $d^{-\frac{d-3}{2}}\hat{\delta}_d^{-\frac{d-1}{d}}>d^{-\frac{d-3}{2}}>\frac{1}{2}d^{-\frac{d-1}{2}}$ hold for all $d\geq 3$, and therefore (\ref{Main-1}) implies (resp., significantly improves) (\ref{BeSzSz-Main}) for all (resp., sufficiently large) $n>1$.
\end{remark}

\begin{remark}\label{Rogers}
Recall the following classical result of Rogers \cite{Ro} (which was rediscovered by Baranovskii \cite{Ba} and extended to
spherical and hyperbolic spaces by B\"or\"oczky \cite{Bo}): Let $\mathcal {P}:=\{\B^d[\ccc_i,1]| i\in I\}$ be an arbitrary packing of unit balls in $\Ee^d$, $d>1$ with $\mathbf{V}_i$ standing for the Voronoi cell assigned to $\B^d[\ccc_i,1]$ for $i\in I$. Furthermore, take a regular $d$-dimensional simplex of edge length $2$ in $\Ee^d$ and then draw a $d$-dimensional unit ball around each vertex of the simplex. Finally, let $\sigma_d$ denote the ratio of the volume of the portion of the simplex covered by balls to the volume of the simplex. Then $\frac{\omega_d}{{\rm vol}_d\left(\mathbf{V}_i\cap \B^d[\ccc_i,\sqrt{\frac{2d}{d+1}}]\right)}\leq \sigma_d$ holds for all $i\in I$ and therefore $\hat{\delta}_d\leq \sigma_d$ for all $d\geq 3$. The latter inequality and (\ref{Main-1}) yield that if $\mathcal {P}$ is an arbitrary LS-packing of $n>1$ unit balls in $\Ee^d$, $d\geq 3$, then
$c(\mathcal {P})\leq \left\lfloor dn-\left(d^{-\frac{d-3}{2}}\hat{\delta}_d^{-\frac{d-1}{d}}\right)n^{\frac{d-1}{d}} \right\rfloor \leq \left\lfloor dn-\left(d^{-\frac{d-3}{2}}{\sigma}_d^{-\frac{d-1}{d}}\right)n^{\frac{d-1}{d}} \right\rfloor$, where $\sigma_d\sim\frac{d}{e}2^{-\frac{1}{2}d}$ (\cite{Ro}).
\end{remark}

\begin{remark}\label{Bezdek}
We note that the density upper bound $\sigma_d$ of Rogers has been improved by the author \cite{Bez02} for dimensions $d\geq 8$ as follows: Using the notations of Remark~\ref{Rogers}, \cite{Bez02} shows that  $\frac{\omega_d}{{\rm vol}_d\left(\mathbf{V}_i\cap \B^d[\ccc_i,\sqrt{\frac{2d}{d+1}}]\right)}\leq \hat{\sigma}_d$ holds for all $i\in I$ and $d\geq 8$ and therefore $\hat{\delta}_d\leq \hat{\sigma}_d$ for all $d\geq 8$, where $\hat{\sigma}_d$ is a geometrically well-defined quantity satisfying the inequality $\hat{\sigma}_d<\sigma_d$ for all $d\geq 8$. This result and (\ref{Main-1}) yield that if $\mathcal {P}$ is an arbitrary LS-packing of $n>1$ unit balls in $\Ee^d$, $d\geq 8$, then
$
c(\mathcal {P})\leq \left\lfloor dn-\left(d^{-\frac{d-3}{2}}\hat{\delta}_d^{-\frac{d-1}{d}}\right)n^{\frac{d-1}{d}} \right\rfloor \leq \left\lfloor dn-\left(d^{-\frac{d-3}{2}}{\hat{\sigma}}_d^{-\frac{d-1}{d}}\right)n^{\frac{d-1}{d}} \right\rfloor 
$.
\end{remark}

\begin{remark}\label{Hales}
The density upper bound $\sigma_3$ of Rogers has been improved by Hales \cite{Ha} as follows:  If $\mathcal {P}:=\{\B^3[\ccc_i,1]| i\in I\}$ is an arbitrary packing of unit balls in $\Ee^3$ and $\mathbf{V}_i$ denotes the Voronoi cell assigned to $\B^d[\ccc_i,1]$, $i\in I$, then $\frac{\omega_3}{{\rm vol}_3\left(\mathbf{V}_i\cap \B^3[\ccc_i,\sqrt{2}]\right)}\leq \frac{\omega_3}{{\rm vol}_3(\mathbf{D})}< 0.7547<\sigma_3=0.7797...$, where $\mathbf{D}$ stands for a regular dodacahedron of inradius $1$. Hence, $\hat{\delta}_3< 0.7547$. The latter inequality and (\ref{Main-1}) yield that if $\mathcal {P}$ is an arbitrary LS-packing of $n>1$ unit balls in $\Ee^3$, then
$c(\mathcal {P})\leq \left\lfloor 3n-\hat{\delta}_3^{-\frac{2}{3}}n^{\frac{2}{3}} \right\rfloor\leq\left\lfloor  3n-1.206 n^{\frac{2}{3}}\right\rfloor$.
\end{remark}


\begin{remark}\label{lattice-LS-packing}
Let $\mathcal {P}:=\{\B^d[\ccc_i,\frac{1}{2}]| \ccc_i\in \Ze^d, 1\leq i\leq n\}$ be an arbitrary packing of $n$ unit diameter balls with centers having integer coordinates in $\Ee^d$. Clearly, $\mathcal {P}$ is a TS-packing. Then let $c_{ \Ze^d}(n)$ denote the largest $c(\mathcal {P})$ for packings $\mathcal {P}$ of $n$ unit diameter balls of  $\Ee^d$ obtained in this way. It is proved in \cite{BeSzSz} that 
\begin{equation}\label{integer-lattice}
dN^d-dN^{d-1}\leq c_{ \Ze^d}(n)\leq \left\lfloor dn-dn^{\frac{d-1}{d}}\right\rfloor
\end{equation}
for $N\in\Ze$ satisfying $0\leq N\leq n^{\frac{1}{d}}< N+1$, where $d>1$ and $n>1$. Note that if $N=n^{\frac{1}{d}}\in \Ze$, then the lower and upper estimates of (\ref{integer-lattice}) are equal to $c_{ \Ze^d}(n)$. Furthermore,
\begin{equation}\label{integer-lattice-2D} 
c_{ \Ze^2}(n)=\lfloor 2n-2\sqrt{n}\rfloor
\end{equation} 
for all $n>1$. We note that \cite{Ne} (resp., \cite{AlCe}) generates an algorithm that lists some (resp., all) packings  $\mathcal {P}=\{\B^d[\ccc_i,\frac{1}{2}]| \ccc_i\in \Ze^d, 1\leq i\leq n\}$ with $c(\mathcal {P})=c_{ \Ze^d}(n)$ for $d\ge 4$ (resp., $d=2, 3$) and $n>1$. 
\end{remark}

Our second main result is a solution of the contact number problem for LS-packings of congruent disks in $\Ee^2$. As in Remark~\ref{lattice-LS-packing}, it will be convenient for us to study LS-packings of unit diameter disks instead of unit disks (Figure~\ref{Figure1.pdf}).

\begin{figure}[ht]
\begin{center}
\includegraphics[width=0.5\textwidth]{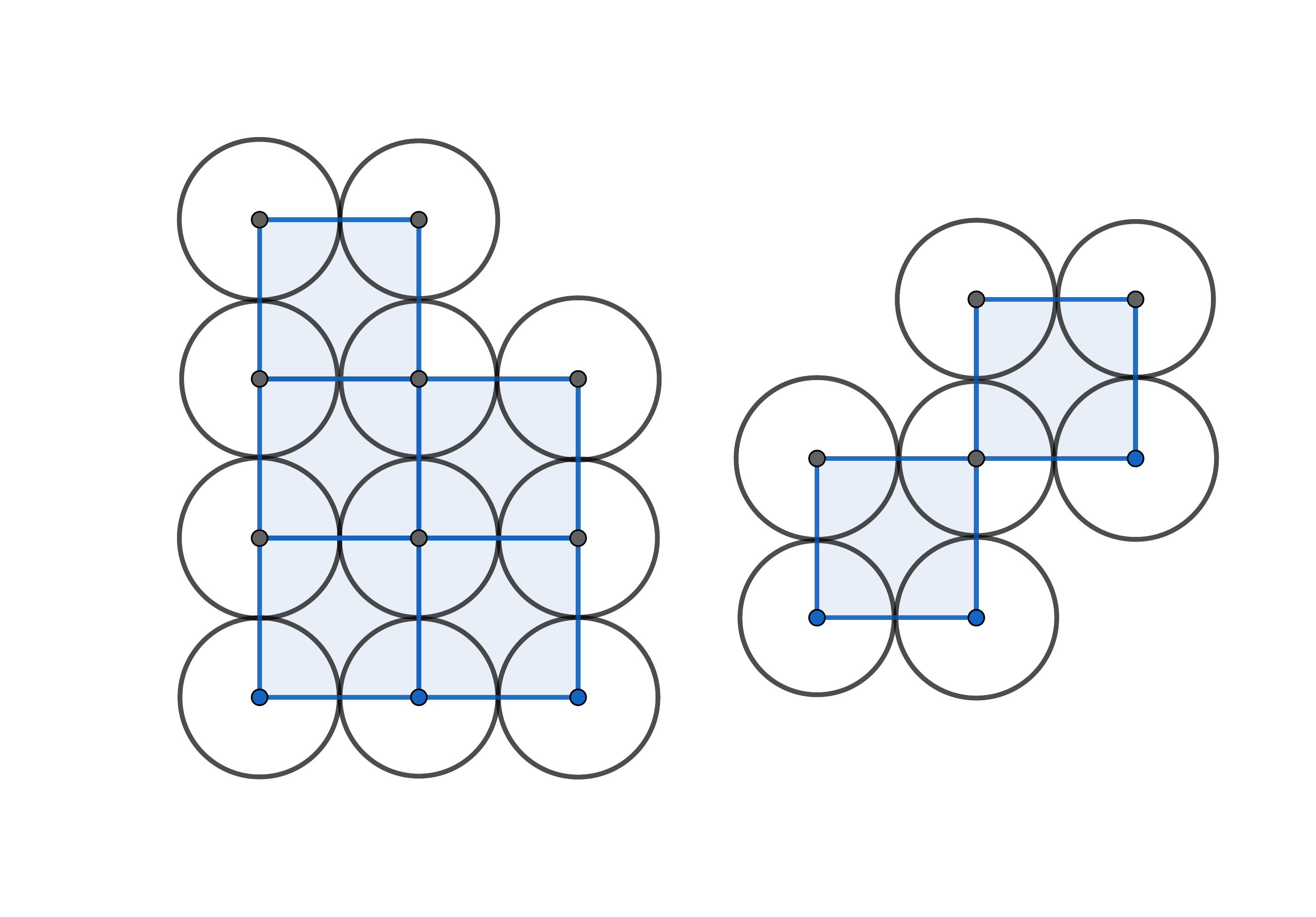}
\caption{An LS-packing of $11$ (resp., $7$) unit diameter disks with maximum contact number.}
\label{Figure1.pdf} 
\end{center}
\end{figure}

\begin{theorem}\label{2D-theorem}
Let $\mathcal {P}$ be an arbitrary LS-packing of $n>1$ unit diameter disks in $\Ee^2$. Then $c(\mathcal {P})\leq \lfloor 2n-2\sqrt{n}\rfloor$.
\end{theorem}
Recall (\cite{BeSzSz}) that $c_{TS}(n, 2)=\lfloor 2n-2\sqrt{n}\rfloor$ for all $n>1$. Clearly, this result, Remark~\ref{lattice-LS-packing}, and Theorem~\ref{2D-theorem} imply
\begin{corollary}\label{2D-corollary-1}
$c_{ \Ze^2}(n)=c_{TS}(n, 2)=c_{LS}(n, 2)=\lfloor 2n-2\sqrt{n}\rfloor$ for all $n>1$.
\end{corollary}


In addition, we have the following statement, which is an LS-packing analogue of the crystallization result of \cite{HeRa} and characterizes all LS-packings of $n$ unit disks having maximum contact number. (See \cite{AlCe} for an algorithm that gives a complete list for a given $n>1$ of the extremal polyominoes discussed in Theorem~\ref{2D-crystallization}.)

\begin{theorem}\label{2D-crystallization}  
Suppose that $\mathcal {P}$ is an LS-packing of $n$ unit diameter disks with $c(\mathcal {P})= \lfloor 2n-2\sqrt{n}\rfloor$, $n\geq 4$ in $\Ee^2$. Let $G_c(\mathcal {P})$ denote the contact graph of $\mathcal {P}$ embedded in $\Ee^2$ such that the vertices are the center points of the unit diameter disks of $\mathcal {P}$ and the edges are line segments of unit length each connecting two center points if and only if the unit diameter disks centered at those two points touch each other. Then either $G_c(\mathcal {P})$ is the contact graph of the LS-packing of $7$ unit diameter disks shown in Figure~\ref{Figure1.pdf} or
{\item (i)} $ G_c(\mathcal {P})$ is $2$-connected whose internal faces (i.e., faces different from its external face) form an edge-to-edge connected family of unit squares called a polyomino of an isometric copy of the integer lattice $ \Ze^2$ in $\Ee^2$ (see the first packing in Figure~\ref{Figure1.pdf}) or
{\item (ii)} $G_c(\mathcal {P})$ is $2$-connected whose internal faces are unit squares forming a polyomino of an isometric copy of the integer lattice $ \Ze^2$ in $\Ee^2$ with the exception of one internal face which is a pentagon adjacent along (at least) three consecutive sides to the external face of $ G_c(\mathcal {P})$ and along (at most) two consecutive sides to the polyomino (see the second packing in Figure~\ref{Figure2.pdf}) or 
{\item (iii)}  $G_c(\mathcal {P})$ possesses a degree one vertex on the boundary of its external face such that deleting that vertex together with the edge adjacent to it yields a $2$-connected graph whose internal faces are unit squares forming a polyomino of an isometric copy of the integer lattice $ \Ze^2$ in $\Ee^2$ (see the first packing in Figure~\ref{Figure2.pdf}). 
\end{theorem} 

\begin{figure}[ht]
\begin{center}
\includegraphics[width=0.5\textwidth]{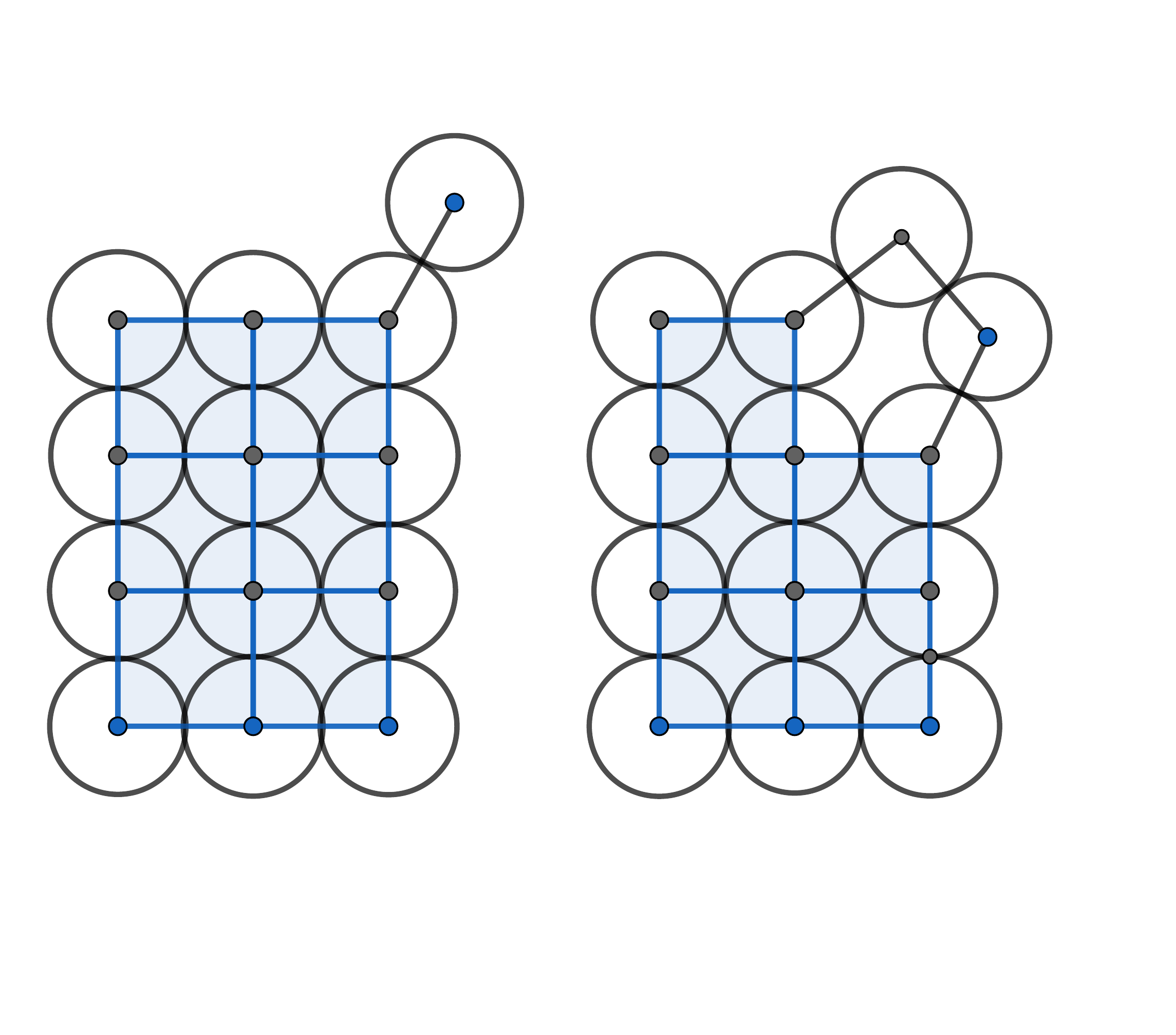}
\caption{Two LS-packings of $13$ unit diameter disks with maximum contact number, the second of which is not a TS-packing.}
\label{Figure2.pdf} 
\end{center}
\end{figure}

The following crystallization result of TS-packings follows from Theorem~\ref{2D-crystallization} in a straightforward way.

\begin{corollary}\label{2D-corollary-2}
Suppose that $\mathcal {P}$ is an TS-packing of $n$ unit diameter disks with $c(\mathcal {P})= \lfloor 2n-2\sqrt{n}\rfloor$, $n\geq 4$ in $\Ee^2$. Then either $G_c(\mathcal {P})$ is the contact graph of the LS-packing of $7$ unit diameter disks shown in Figure~\ref{Figure1.pdf} or
{\item (i)} $ G_c(\mathcal {P})$ is $2$-connected whose internal faces are unit squares forming a polyomino of an isometric copy of the integer lattice $ \Ze^2$ in $\Ee^2$ (see the first packing in Figure~\ref{Figure1.pdf}) or
{\item (ii)}  $G_c(\mathcal {P})$ possesses a degree one vertex on the boundary of its external face such that deleting that vertex together with the edge adjacent to it yields a $2$-connected graph whose internal faces are unit squares forming a polyomino of an isometric copy of the integer lattice $ \Ze^2$ in $\Ee^2$ (see the first packing in Figure~\ref{Figure2.pdf}). 
\end{corollary}



In the rest of the paper we prove the theorems stated.

\section{Proof of Theorem~\ref{Main-theorem}}

The following proof of Theorem~\ref{Main-theorem} is quite different from the proof of (\ref{BeSzSz-Main}) published in \cite{BeSzSz} although both are of volumetric nature. In particular, the spherical density estimate of Lemma~\ref{spherical-density-bound} and the surface volume estimates in terms of the contact number of the given LS-packing of unit balls stated under (\ref{inequality-8}), (\ref{inequality-9}), (\ref{equality-10}) and (\ref{inequality-11}) represent those new ideas that lead to a proof of Theorem~\ref{Main-theorem} and so, to an improvement of the upper estimate in (\ref{BeSzSz-Main}).

Let $\mathcal {P}:=\{\B^d[\ccc_i,1]|1\leq i\leq n\}$ be an arbitrary packing of $n>1$ unit balls in $\Ee^d$, $d\geq 3$ and let $\mathbf{V}_i:=\{\x\in \Ee^d | \ \|\x-\ccc_i\|\leq \|\x-\ccc_j\|\ {\rm for \ all}\ j\neq i, 1\leq j\leq n\}$ denote the Voronoi cell assigned to $\B^d[\ccc_i,1]$ for $1\leq i\leq n$. As $\bigcup_{i=1}^n \B^d[\ccc_i,\sqrt{d}]=\bigcup_{i=1}^n\left(\mathbf{V}_i\cap \B^d[\ccc_i,\sqrt{d}]\right)$ therefore it is immediate that
\begin{equation}\label{inequality-1}
\frac{\sum_{i=1}^n{\rm vol}_d\left(\B^d[\ccc_i,1]\right)}{{\rm vol}_d\left(\bigcup_{i=1}^n \B^d[\ccc_i,\sqrt{d}]\right)}=\frac{n\omega_d}{{\rm vol}_d\left(\bigcup_{i=1}^n \B^d[\ccc_i,\sqrt{d}]\right)}\leq\hat{\delta}_d .  
\end{equation}
Next, recall the isoperimetric inequality \cite{Oss} and apply it to $\bigcup_{i=1}^n \B^d[\ccc_i,\sqrt{d}]$:
\begin{equation}\label{inequality-2}
\frac{{\rm svol}_{d-1}^d\left({\rm bd}(\B^d[\oo,1])\right)}{{\rm vol}_d^{d-1}\left(\B^d[\oo,1]\right)}=d^d\omega_d\leq  \frac{{\rm svol}_{d-1}^d\left({\rm bd}\left(\bigcup_{i=1}^n \B^d[\ccc_i,\sqrt{d}]\right)\right)}{{\rm vol}_d^{d-1}\left(\bigcup_{i=1}^n \B^d[\ccc_i,\sqrt{d}]\right)} ,
\end{equation}
where ${\rm bd}(\cdot)$ refers to the boundary of the given set in $\Ee^d$ and ${\rm svol}_{d-1}(\cdot)$ denotes the $(d-1)$-dimensional surface volume of the boundary of the given set in $\Ee^d$. Hence, (\ref{inequality-1}) and (\ref{inequality-2}) imply in a straightforward way that
\begin{equation}\label{inequality-3}
\left(d\omega_d\hat{\delta}_d^{-\frac{d-1}{d}}\right)n^{\frac{d-1}{d}}\leq {\rm svol}_{d-1}\left({\rm bd}\left(\bigcup_{i=1}^n \B^d[\ccc_i,\sqrt{d}]\right)\right) .
\end{equation}

In the next part of the proof of Theorem~\ref{Main-theorem}, we investigate the vertices of the contact graph $G_c(\mathcal {P})$ of the LS-packing $\mathcal {P}=\{\B^d[\ccc_i,1]|1\leq i\leq n\}$ from combinatorial as well as volumetric point of view. The details are as follows. Let $T_i:=\{ j | 1\leq j\leq n\ {\rm and}\ \|\ccc_j-\ccc_i\|=2\}$ for $1\leq i\leq n$. Clearly, ${\rm card}(T_i)$ is the degree of the vertex of $G_c(\mathcal {P})$ corresponding to the packing element $\B^d[\ccc_i,1]$. Now, let $\uu_{ij}:=\frac{1}{2}(\ccc_j-\ccc_i)\in \Sedm$ for $j\in T_i$ and $1\leq i\leq n$. As $\mathcal {P}$ is an LS-packing therefore it is easy to show that the angular distance between any two points of $\{\uu_{ij} | j\in T_i\}\subset \Sedm$ is at least $\frac{\pi}{2}$, where $1\leq i\leq n$. Thus, a result of Davenport and Haj\'os \cite{DaHa}, proved independently by Rankin \cite{Ran} (see also \cite{Ku}), implies in a straightforward way that
\begin{equation}\label{Davenport-Hajos} 
{\rm card}(T_i)\leq 2d
\end{equation} 
holds for all $1\leq i\leq n$. From this it follows immediately that $c(\mathcal {P})\leq dn$. We will significantly improve this estimate with the help of the volumetric inequality of Lemma~\ref{spherical-density-bound}. We need the following notations: Let $S^{d-1}(\p , r):={\rm bd }\left(\B^d[\p,r]\right)$ denote the $(d-1)$-dimensional sphere centered at $\p$ having radius $r>0$ in $\Ee^d$ and let $C_{S^{d-1}(\p , r)}[\x,\alpha]:=\{\y\in  S^{d-1}(\p , r) | \langle \x-\p ,\y-\p \rangle\geq r^2 \cos\alpha\}$ denote the closed $(d-1)$-dimensional spherical cap centered at $\x\in  S^{d-1}(\p , r)$ having angular radius $0\leq \alpha\leq\pi$ in $ S^{d-1}(\p , r)$. We note that $S^{d-1}(\oo , 1)=\Sedm$. Furthermore, let $V_{S^{d-1}(\p , r)}(\cdot)$ denote the spherical Lebesgue measure in $S^{d-1}(\p , r)$ with $V_{S^{d-1}(\p , r)}(S^{d-1}(\p , r))=d\omega_dr^{d-1}$. 

\begin{lemma}\label{spherical-density-bound} If $1\leq i\leq n$ and $j\in T_i$, then 
\begin{equation}\label{equation-4}
S^{d-1}(\ccc_i , \sqrt{d})\cap\B^d[\ccc_j,\sqrt{d}]=C_{S^{d-1}(\ccc_i , \sqrt{d})}\left[\ccc_i+\sqrt{d}\uu_{ij},\arccos\frac{1}{\sqrt{d}}\right] .
\end{equation}  
Furthermore, the spherical caps $\left\{C_{S^{d-1}(\ccc_i , \sqrt{d})}\left[\ccc_i+\sqrt{d}\uu_{ij},\frac{\pi}{4}\right] \bigg| j\in T_i\right\}$ form a packing in $S^{d-1}(\ccc_i , \sqrt{d})$
for all $1\leq i\leq n$. In particular,
\begin{equation}\label{inequality-5}
\frac{\sum_{j\in T_i}V_{S^{d-1}(\ccc_i , \sqrt{d})}\left(C_{S^{d-1}(\ccc_i , \sqrt{d})}\left[\ccc_i+\sqrt{d}\uu_{ij},\frac{\pi}{4}\right]\right)}{V_{S^{d-1}(\ccc_i , \sqrt{d})}\left(\bigcup_{j\in T_i}C_{S^{d-1}(\ccc_i , \sqrt{d})}\left[\ccc_i+\sqrt{d}\uu_{ij},\arccos\frac{1}{\sqrt{d}}\right]\right)}\leq\frac{2}{\omega_d}V_{\Sedm}\left(C_{\Sedm}\left[\x ,\frac{\pi}{4}\right]\right)
\end{equation}
holds for all $1\leq i\leq n$ and any $\x\in  \Sedm$.
\end{lemma}

\begin{figure}[ht]
\begin{center}
\includegraphics[width=0.55\textwidth]{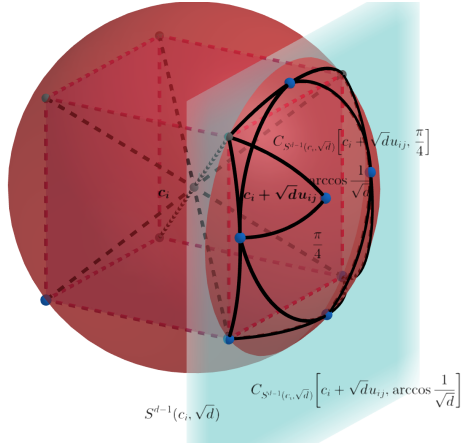}
\caption{The spherical caps  $\left\{C_{S^{d-1}(\ccc_i , \sqrt{d})}\left[\ccc_i+\sqrt{d}\uu_{ij},\frac{\pi}{4}\right] | j\in T_i\right\}$ form a packing in $S^{d-1}(\ccc_i , \sqrt{d})$.}
\label{Figure3.pdf} 
\end{center}
\end{figure}

\begin{proof} Clearly $1\leq i\leq n$ and  $j\in T_i$ imply$ \|\ccc_j-\ccc_i\|=2$ and so, a straightforward computation yields (\ref{equation-4}).  On the other hand, $\mathcal {P}$ is an LS-packing and therefore from the discussion of proving (\ref{Davenport-Hajos}) it follows that $\{C_{\Sedm}\left[\uu_{ij},\frac{\pi}{4}\right] | j\in T_i\}$ is a packing in $\Sedm$ for all $1\leq i\leq n$ and so, also the spherical caps $\left\{C_{S^{d-1}(\ccc_i , 1)}\left[\ccc_i+\uu_{ij},\frac{\pi}{4}\right] | j\in T_i\right\}$ form a packing in $S^{d-1}(\ccc_i , 1)$ for all $1\leq i\leq n$. This implies via central projection that the spherical caps  $\left\{C_{S^{d-1}(\ccc_i , \sqrt{d})}\left[\ccc_i+\sqrt{d}\uu_{ij},\frac{\pi}{4}\right] | j\in T_i\right\}$ form a packing in $S^{d-1}(\ccc_i , \sqrt{d})$ for all $1\leq i\leq n$ (Figure~\ref{Figure3.pdf}). Using central projection again, one obtains that
\begin{equation}\label{equation-6}
\frac{\sum_{j\in T_i}V_{S^{d-1}(\ccc_i , \sqrt{d})}\left(C_{S^{d-1}(\ccc_i , \sqrt{d})}\left[\ccc_i+\sqrt{d}\uu_{ij},\frac{\pi}{4}\right]\right)}{V_{S^{d-1}(\ccc_i , \sqrt{d})}\left(\bigcup_{j\in T_i}C_{S^{d-1}(\ccc_i , \sqrt{d})}\left[\ccc_i+\sqrt{d}\uu_{ij},\arccos\frac{1}{\sqrt{d}}\right]\right)}=\frac{\sum_{j\in T_i}V_{\Sedm}\left(C_{\Sedm}\left[\uu_{ij},\frac{\pi}{4}\right]\right)}{V_{\Sedm}\left(\bigcup_{j\in T_i}C_{\Sedm}\left[\uu_{ij},\arccos\frac{1}{\sqrt{d}}\right]\right)} 
\end{equation}
holds for all $1\leq i\leq n$. Next, let $\mathbf{V}_{ij}(\Sedm):=\{\y\in\Sedm | \langle\uu_{ij}, \y\rangle\geq \langle\uu_{ik}, \y\rangle\}\ {\rm for}\ {\rm all}\ k\neq j, k\in T_i\}$ denote the spherical Voronoi cell assigned to $C_{\Sedm}\left[\uu_{ij},\frac{\pi}{4}\right]$ in the packing $\{C_{\Sedm}\left[\uu_{ij},\frac{\pi}{4}\right] | j\in T_i\}$ of $\Sedm$, where $1\leq i\leq n$. Clearly, $\bigcup_{j\in T_i}C_{\Sedm}\left[\uu_{ij},\arccos\frac{1}{\sqrt{d}}\right]=\bigcup_{j\in T_i}\left(C_{\Sedm}\left[\uu_{ij},\arccos\frac{1}{\sqrt{d}}\right]\cap \mathbf{V}_{ij}(\Sedm) \right)$ for all $1\leq i\leq n$ and so,
\begin{equation}\label{inequality-7}
\frac{\sum_{j\in T_i}V_{\Sedm}\left(C_{\Sedm}\left[\uu_{ij},\frac{\pi}{4}\right]\right)}{V_{\Sedm}\left(\bigcup_{j\in T_i}C_{\Sedm}\left[\uu_{ij},\arccos\frac{1}{\sqrt{d}}\right]\right)} \leq\max_{j\in T_i}\frac{V_{\Sedm}\left(C_{\Sedm}\left[\uu_{ij},\frac{\pi}{4}\right]\right)}{V_{\Sedm}\left( C_{\Sedm}\left[\uu_{ij},\arccos\frac{1}{\sqrt{d}}\right]\cap \mathbf{V}_{ij}(\Sedm)\right)}
\end{equation}
holds for all $1\leq i\leq n$. On the one hand, notice that $\arccos\frac{1}{\sqrt{d}}$ is the (angular) circumradius of the regular $(d-1)$-dimensional spherical simplex $\Delta_{\Sedm}(\frac{\pi}{2})$ of edge length $\frac{\pi}{2}$ in $\Sedm$. On the other hand, if we denote the ratio of the spherical volume of the portion of $\Delta_{\Sedm}(\frac{\pi}{2})$ covered by the pairwise tangent $(d-1)$-dimensional spherical caps centered at the vertices of $\Delta_{\Sedm}(\frac{\pi}{2})$ having angular radius $\frac{\pi}{4}$ to the spherical volume of $\Delta_{\Sedm}(\frac{\pi}{2})$ by $\sigma_{\Sedm}(\frac{\pi}{2})$, then a straighforward computation (based on the observation that $2^d$ isometric copies of  $\Delta_{\Sedm}(\frac{\pi}{2})$ tile $\Sedm$), yields
\begin{equation} \label{Boroczky-density-bound}
\sigma_{\Sedm}\left(\frac{\pi}{2}\right)=\frac{2}{\omega_d}V_{\Sedm}\left(C_{\Sedm}\left[\x ,\frac{\pi}{4}\right]\right)
\end{equation}
for any $\x\in\Sedm$. Finally, B\"or\"oczky's upper bound theorem \cite{Bo} for the density of a spherical cap in its properly truncated spherical Voronoi cell (see pages 258-260 in \cite{Bo}) implies via a standard limiting process that
\begin{equation}\label{Boroczky-upper-bound-theorem}
\frac{V_{\Sedm}\left(C_{\Sedm}\left[\uu_{ij},\frac{\pi}{4}\right]\right)}{V_{\Sedm}\left( C_{\Sedm}\left[\uu_{ij},\arccos\frac{1}{\sqrt{d}}\right]\cap \mathbf{V}_{ij}(\Sedm)\right)}\leq 
\sigma_{\Sedm}\left(\frac{\pi}{2}\right)
\end{equation}
holds for all $1\leq i\leq n$ and $j\in T_i$. Thus, (\ref{equation-6}), (\ref{inequality-7}), (\ref{Boroczky-density-bound}), and (\ref{Boroczky-upper-bound-theorem}) yield (\ref{inequality-5}) in a straightforward way, finishing the proof of Lemma~\ref{spherical-density-bound}.
\end{proof}
We proceed by properly upper estimating ${\rm svol}_{d-1}\left({\rm bd}\left(\bigcup_{i=1}^n \B^d[\ccc_i,\sqrt{d}]\right)\right)$ for the LS-packing $\mathcal {P}=\{\B^d[\ccc_i,1]|1\leq i\leq n\}$, where $d\geq 3$ and $n>1$. First, notice that (\ref{equation-4}) in Lemma~\ref{spherical-density-bound} easily yields
$${\rm svol}_{d-1}\left({\rm bd}(\cup_{i=1}^n \B^d[\ccc_i,\sqrt{d}])\right) $$
\begin{equation}\label{inequality-8}
\leq \sum_{i=1}^{n}\left(\omega_dd^{\frac{d+1}{2}} -  V_{S^{d-1}(\ccc_i , \sqrt{d})}\left(\bigcup_{j\in T_i}C_{S^{d-1}(\ccc_i , \sqrt{d})}\left[\ccc_i+\sqrt{d}\uu_{ij},\arccos\frac{1}{\sqrt{d}}\right]\right)\right) .
\end{equation}
Second, (\ref{inequality-5}) in Lemma~\ref{spherical-density-bound} implies
$$ \sum_{i=1}^{n}V_{S^{d-1}(\ccc_i , \sqrt{d})}\left(\bigcup_{j\in T_i}C_{S^{d-1}(\ccc_i , \sqrt{d})}\left[\ccc_i+\sqrt{d}\uu_{ij},\arccos\frac{1}{\sqrt{d}}\right]\right)$$ 
\begin{equation}\label{inequality-9}
\geq \frac{\omega_d}{2V_{\Sedm}\left(C_{\Sedm}\left[\x ,\frac{\pi}{4}\right]\right)}\sum_{i=1}^{n}\sum_{j\in T_i}V_{S^{d-1}(\ccc_i , \sqrt{d})}\left(C_{S^{d-1}(\ccc_i , \sqrt{d})}\left[\ccc_i+\sqrt{d}\uu_{ij},\frac{\pi}{4}\right]\right) ,
\end{equation}
where $\x\in  \Sedm$. Finally, according to Lemma~\ref{spherical-density-bound}, the spherical caps $\left\{C_{S^{d-1}(\ccc_i , \sqrt{d})}\left[\ccc_i+\sqrt{d}\uu_{ij},\frac{\pi}{4}\right] \bigg| j\in T_i\right\}$ form a packing in $S^{d-1}(\ccc_i , \sqrt{d})$ for all $1\leq i\leq n$ and so,
\begin{equation}\label{equality-10}
\sum_{i=1}^{n}\sum_{j\in T_i}V_{S^{d-1}(\ccc_i , \sqrt{d})}\left(C_{S^{d-1}(\ccc_i , \sqrt{d})}\left[\ccc_i+\sqrt{d}\uu_{ij},\frac{\pi}{4}\right]\right)=2d^{\frac{d-1}{2}}V_{\Sedm}\left(C_{\Sedm}\left[\x ,\frac{\pi}{4}\right]\right)c(\mathcal {P}) 
\end{equation}
holds for any $\x\in  \Sedm$. Hence, (\ref{inequality-8}), (\ref{inequality-9}), and (\ref{equality-10}) yield
\begin{equation}\label{inequality-11}
{\rm svol}_{d-1}\left({\rm bd}(\cup_{i=1}^n \B^d[\ccc_i,\sqrt{d}])\right)\leq \omega_dd^{\frac{d+1}{2}}n- \omega_dd^{\frac{d-1}{2}}c(\mathcal {P}) .
\end{equation}

Finally, as a last step in the proof of Theorem~\ref{Main-theorem}, we combine (\ref{inequality-3}) and (\ref{inequality-11}) to get 
\begin{equation}\label{inequality-12}
\left(d\omega_d\hat{\delta}_d^{-\frac{d-1}{d}}\right)n^{\frac{d-1}{d}}\leq \omega_dd^{\frac{d+1}{2}}n- \omega_dd^{\frac{d-1}{2}}c(\mathcal {P}) ,
\end{equation}
from which (\ref{Main-1}) follows in a straightforward way. This completes the proof of Theorem~\ref{Main-theorem}.

\section{Proof of Theorem~\ref{2D-theorem}}

We modify Harborth's method (\cite{Har}) somewhat and apply it to locally separable congruent disk packings as follows. Let $\mathcal {P}$ be an arbitrary LS-packings of $n>1$ unit diameter disks in $\Ee^2$. Moreover, let $G_c(\mathcal {P})$ denote the contact graph of $\mathcal {P}$ embedded in $\Ee^2$ such that the vertices are the center points of the unit diameter disks of $\mathcal {P}$ and the edges are line segments of unit length each connecting two center points if and only if the unit diameter disks centered at those two points touch each other. 
Our goal is to show by induction on $n$ that 
\begin{equation}\label{inequality-13}
c(\mathcal {P})\leq \lfloor 2n-2\sqrt{n}\rfloor .
\end{equation}
It is easy to check that (\ref{inequality-13}) holds for $n=2, 3$. So, by induction we assume that (\ref{inequality-13}) holds for any LS-packing of at most $n-1$ unit diameter disks in $\Ee^2$, where $n-1\geq 3$. Next, let $\mathcal {P}$ be an arbitrary  LS-packing of $n\geq 4$ unit diameter disks in $\Ee^2$. We claim that 
one may assume that $G_c(\mathcal {P})$ is $2$-connected (and therefore the degree of each vertex of $G_c(\mathcal {P})$ is at least $2$). Namely, assume that $G_c(\mathcal {P})$ possesses two subgraphs $G_1$ and $G_2$ with only one vertex in common and with $n_1\geq 2$ and $n_2\geq 2$ vertices, respectively. Then, $n+1=n_1+n_2$, and by induction the number of edges of $G_c(\mathcal {P})$ is at most $(2n_1-2\sqrt{n_1})+(2n_2-2\sqrt{n_2})$, which can be easily upper bounded by $2n-2\sqrt{n}$ implying (\ref{inequality-13}) in a straightforward way. Thus, from this point on we assume that $G_c(\mathcal {P})$ is a $2$-connected planar graph having $c_{\rm LS}(n)$ edges, where $c_{\rm LS}(n)$ stands for the largest number of edges of contact graphs of LS-packings of $n$ unit diameter disks in $\Ee^2$. Hence, every face of $G_c(\mathcal {P})$ - including the external one - is bounded by a cycle. 

Let $P$ be the simple closed polygon that bounds the external face of $G_c(\mathcal {P})$ and let $v$ denote the number of vertices of $P$. As $\mathcal {P}$ is an LS-packing therefore the degree of every vertex of $G_c(\mathcal {P})$ is at most $4$ and in particular, the measure of the internal angle of $P$ at a vertex of degree $j$ is at least as large as $\frac{(j-1)\pi}{2}$, where $2\leq j\leq 4$. If $v_j$ denotes the number of vertices of $P$ of degree $j$ for $2\leq j\leq 4$, then
\begin{equation}\label{equality-14}
v=v_2+v_3+v_4 .
\end{equation}
Moreover, as the sum of the measures of the internal angles of $P$ is equal to $(v-2)\pi$ therefore $\sum_{j=2}^{4}v_j\frac{(j-1)\pi}{2}\leq (v-2)\pi$ implying that
\begin{equation}\label{inequality-15}
v_2+2v_3+3v_4\leq 2v-4 .
\end{equation}
Next, delete the vertices of $P$ from the vertices of $G_c(\mathcal {P})$ together with the edges that are incident to them. As a result we get that $c_{\rm LS}(n)-v-(v_3+2v_4))\leq c_{\rm LS}(n-v)$, which together with (\ref{equality-14}) and (\ref{inequality-15}) imply that
\begin{equation}\label{inequality-16}
c_{\rm LS}(n)\leq c_{\rm LS}(n-v)+2v-4 .
\end{equation}
Finally, (\ref{inequality-16}) combined with the induction hypothesis $c_{\rm LS}(n-v)\leq 2(n-v)-2\sqrt{n-v}$ yields
\begin{equation}\label{inequality-17}
c_{\rm LS}(n)\leq (2n-4)-2\sqrt{n-v} .
\end{equation}

In order to finish our inductive proof we are going to upper bound the right-hand side of (\ref{inequality-17}) by lower bounding $n-v$ as follows. Let $f_k$ denote the number of internal faces of $G_c(\mathcal {P})$ that have $k$ sides each having unit length. (Recall that an internal face means a face of $G_c(\mathcal {P})$ different from its external face.) As $\mathcal {P}$ is an LS-packing therefore $k\geq 4$. Euler's formula for $G_c(\mathcal {P})$ implies that
\begin{equation}\label{equality-18}
f_4+f_5+\dots =c_{\rm LS}(n)-n+1
\end{equation}
On the other hand, it is obvious that
\begin{equation}\label{inequality-19}
4(f_4+f_5+\dots )\leq 4f_4+5f_5+\dots =v+2(c_{\rm LS}(n)-v) .
\end{equation} 
Hence, (\ref{equality-18}) and (\ref{inequality-19}) yield
\begin{equation}\label{inequality-20}
2c_{\rm LS}(n)-3n+4\leq n-v ,
\end{equation}
which is the desired lower bound for $n-v$. 

Now, (\ref{inequality-17}) and (\ref{inequality-20}) imply $c_{\rm LS}(n)\leq (2n-4)-2\sqrt{2c_{\rm LS}(n)-3n+4}$, which is equivalent to
\begin{equation}\label{inequality-21}
c^2_{\rm LS}(n)-4n c_{\rm LS}(n)+(4n^2-4n)\geq 0 .
\end{equation}
As the roots of the quadratic equation $x^2-4nx+(4n^2-4n)=0$ are $x_{1,2}=2n\pm 2\sqrt{n}$ and $c_{\rm LS}(n)<2n$, it follows that $c_{\rm LS}(n)\leq 2n-2\sqrt{n}$, finishing the proof of Theorem~\ref{2D-theorem}.

\section{Proof of Theorem~\ref{2D-crystallization}}

The following case analytic proof of Theorem~\ref{2D-crystallization} is based on the above proof of Theorem~\ref{2D-theorem}. We start with

\begin{proposition}\label{proposed-approach-1}
Let $\mathcal {P}$ be an LS-packing of $n>1$ unit diameter disks with $c(\mathcal {P})= \lfloor 2n-2\sqrt{n}\rfloor$ such that $G_c(\mathcal {P})$ is $2$-connected. If the simple closed polygon that bounds the external face of $G_c(\mathcal {P})$ has $v$ vertices, then the inequalities
\begin{equation}\label{inequality-22}
c(\mathcal {P})+1\leq (2n-4)-2\sqrt{n-v} ,
\end{equation}
\begin{equation}\label{inequality-23}
2(c(\mathcal {P})+1)-3n+4\leq n-v
\end{equation}
cannot hold simultaneously.
\end{proposition}
\begin{proof} Assume (\ref{inequality-22}) and (\ref{inequality-23}). Then it follows that $c(\mathcal {P})+1\leq (2n-4)-2\sqrt{2(c(\mathcal {P})+1)-3n+4}$, which is equivalent to
\begin{equation}\label{inequality-24}
c^2(\mathcal {P})-(4n-2)c(\mathcal {P})+(4n^2-8n+1)\geq 0 .
\end{equation}
As the roots of the quadratic equation $x^2-(4n-2)x+(4n^2-8n+1)=0$ are $x_{1,2}=2n\pm 2\sqrt{n}-1$ and $c(\mathcal {P})<2n$, it follows that $c(\mathcal {P})\leq 2n-2\sqrt{n}-1=\lfloor 2n-2\sqrt{n}\rfloor-1$, a contradiction. This completes the proof of Proposition~\ref{proposed-approach-1}. \end{proof}

Clearly, Proposition~\ref{proposed-approach-1}, (\ref{inequality-17}), and (\ref{inequality-20}) yield
\begin{corollary}\label{proposed-approach-2}
Let $\mathcal {P}$ be an LS-packing of $n>1$ unit diameter disks with $c(\mathcal {P})= \lfloor 2n-2\sqrt{n}\rfloor$ such that $G_c(\mathcal {P})$ is $2$-connected. If the simple closed polygon that bounds the external face of $G_c(\mathcal {P})$ has $v$ vertices, then either
\begin{equation}\label{equation-25}
c(\mathcal {P})=\lfloor (2n-4)-2\sqrt{n-v}\rfloor \ {\text or}
\end{equation}
\begin{equation}\label{equation-26}
2c(\mathcal {P})-3n+4=(n-v)-1 \ or
\end{equation}
\begin{equation}\label{equation-27}
2c(\mathcal {P})-3n+4=n-v .
\end{equation}
\end{corollary}
Now, the discussion that leads to (\ref{inequality-20}) in the proof of Theorem~\ref{2D-theorem} easily implies that if (\ref{equation-27}) holds, then the internal faces of $G_c(\mathcal {P})$ are quadrilaterals having sides of length $1$. Moreover, since $\mathcal {P}$ is an LS-packing therefore it follows that the internal angles of any quadrilateral face are all $\geq\frac{\pi}{2}$ and therefore they are all of measure $\frac{\pi}{2}$. This implies that the quadrilateral faces are all unit squares. Hence, it follows that the vertices of $G_c(\mathcal {P})$ belong to an isometric copy of the integer lattice $ \Ze^2$ in $\Ee^2$ such that Case (i) of Theorem~\ref{2D-crystallization} holds. Next assume that (\ref{equation-26}) holds. This time, the discussion right before (\ref{inequality-20}) in the proof of Theorem~\ref{2D-theorem} implies that the internal faces of $G_c(\mathcal {P})$ are quadrilaterals except one, which is a pentagon labelled by ${\bf P}_0$. Since $\mathcal {P}$ is an LS-packing therefore it follows that the quadrilaterals are unit squares and the sides of ${\bf P}_0$ have length $1$ and its internal angles are all $\geq \frac{\pi}{2}$.
 
\begin{claim}\label{one-internal-vertex}
Let ${\bf x}$ be a vertex of ${\bf P}_0$ which is an internal vertex of $G_c(\mathcal {P})$ (i.e., not a vertex of the external face of $G_c(\mathcal {P})$). Then the internal angle of ${\bf P}_0$ at ${\bf x}$ is of measure $\frac{\pi}{2}$ and there are three unit square faces of $G_c(\mathcal {P})$ meeting at ${\bf x}$.
\end{claim} 
\begin{proof}
$G_c(\mathcal {P})$ has either one or two or three square faces meeting at ${\bf x}$. The first two cases are not possible because then either the sum of the measures of the internal angles of ${\bf P}_0$ is larger than $3\pi$ or ${\bf P}_0$ has a side whose length is larger than $1$. Thus, Claim~\ref{one-internal-vertex} follows.
\end{proof}
\begin{claim}\label{two-internal-vertices}
${\bf P}_0$ cannot have two adjacent vertices which are both internal vertices of $G_c(\mathcal {P})$.
\end{claim}
\begin{proof}
Claim~\ref{one-internal-vertex} and the assumption that ${\bf P}_0$ possesses two adjacent vertices say, ${\bf x}$ and ${\bf y}$ which are both internal vertices of $G_c(\mathcal {P})$ forces ${\bf P}_0$ to have an internal angle of measure $\frac{\pi}{3}$ at the vertex opposite to the side spanned by ${\bf x}$ and ${\bf y}$, a contradiction.
\end{proof}


\begin{claim}\label{exactly-one-internal-vertex}
${\bf P}_0$ possesses exactly one vertex which is an internal vertex of $G_c(\mathcal {P})$.
\end{claim}
\begin{proof}
Claim~\ref{two-internal-vertices} easily implies that ${\bf P}_0$ has either one vertex which is an internal vertex of $G_c(\mathcal {P})$ or it has two non-adjacent vertices say, $\x$ and $\y$ that are internal vertices of $G_c(\mathcal {P})$. Indeed, the latter case is not possible. Namely, if $\z$ denotes the vertex of ${\bf P}_0$ that is adjacent to $\x$ as well as $\y$, then applying Claim~\ref{one-internal-vertex} to $\x$ and $\y$ one obtains that the internal angle of ${\bf P}_0$ at $\z$ is either of measure $\pi$ or of measure $\frac{\pi}{2}$. Then it easily follows that the length of the side of ${\bf P}_0$ opposite to $\z$ is either of length $>1$ or it is of length $0$, a contradiction. This completes the proof of Claim~\ref{exactly-one-internal-vertex}.
\end{proof}

Finally, let $\x_0$ be the vertex of ${\bf P}_0$ that according to Claim~\ref{exactly-one-internal-vertex} is an internal vertex of $G_c(\mathcal {P})$. Then Claim~\ref{one-internal-vertex} implies that there are three unit square faces of $G_c(\mathcal {P})$ meeting at ${\bf x}_0$. It follows that the vertices of $G_c(\mathcal {P})$ with the
exception of the two vertices of ${\bf P}_0$ that are not adjacent to $\x_0$ belong to an isometric copy of the integer lattice $ \Ze^2$ in $\Ee^2$ such that Case (ii) of Theorem~\ref{2D-crystallization} holds.

It remains to investigate the case when (\ref{equation-25}) holds. Then (\ref{equality-14}), (\ref{inequality-15}), (\ref{inequality-16}), and (\ref{inequality-17}) imply that (\ref{equation-25}) holds if and only if
\newline
{\it (a)} the internal angle of $P$ at a vertex of degree $j$ is of measure $\frac{(j-1)\pi}{2}$ for all $2\leq j\leq 4$ and 
\newline
{\it (b)} the graph $G_c$ obtained from $G_c(\mathcal {P})$ by deleting the vertices of $P$ together with the edges that are incident to them, is the contact graph of an LS-packing of $n-v$ unit diameter disks having the largest number of contacts.

\begin{figure}[ht]
\begin{center}
\includegraphics[width=0.6\textwidth]{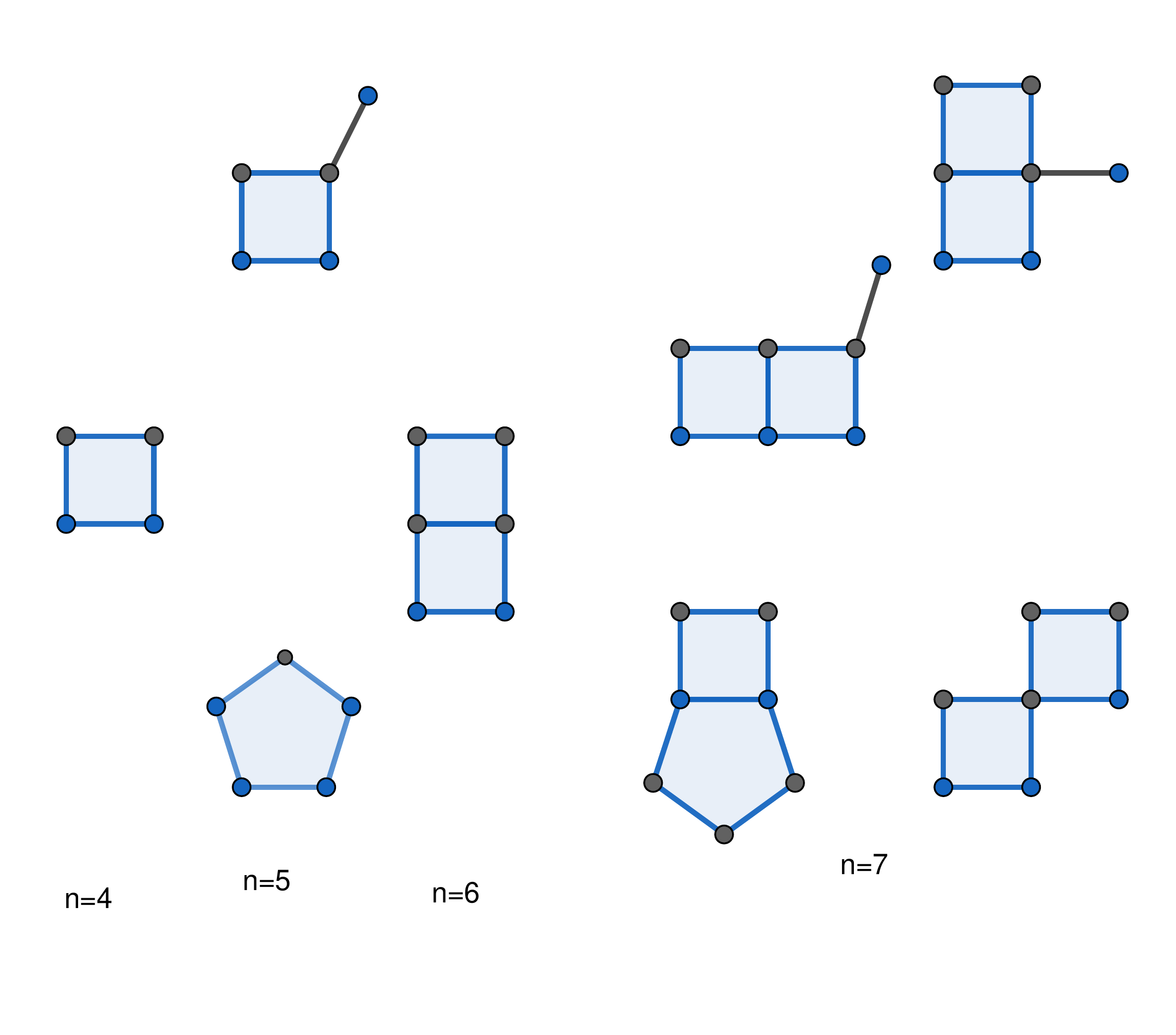}
\caption{List of non-isomorphic contact graphs of LS-packings of $n$ unit diameter disks with maximum contact numbers for $4\leq n\leq 7$.}
\label{Figure4.pdf} 
\end{center}
\end{figure}

As it is easy to check Theorem~\ref{2D-crystallization} for LS-packings of $4\leq n\leq 7$ unit diameter disks (see Figure~\ref{Figure4.pdf}) therefore from this point on we assume that $n\geq 8$. Thus, under {\it (b)} we may assume by induction that $G_c$ possesses the properties of one of the four different cases described in Theorem~\ref{2D-crystallization}. From this and {\it (a)} it follows that the internal faces of $G_c(\mathcal {P})$ are unit squares forming a polyomino of an isometric copy of $ \Ze^2$ in $\Ee^2$ such that Case (i) of Theorem~\ref{2D-crystallization} holds. This completes the proof of the crystallization of LS-packings for the cases (\ref{equation-25}), (\ref{equation-26}), and (\ref{equation-27}).

Finally, we are left to investigate the case when $\mathcal {P}$ is an LS-packing of $n\geq 8$ unit diameter disks with $c(\mathcal {P})= \lfloor 2n-2\sqrt{n}\rfloor$ such that $G_c(\mathcal {P})$ is connected but it is not $2$-connected.

We shall need the following concepts. A {\it horizontal (resp., vertical) bar of length $l\in \Ze_+$ of $ \Ze^2$} is a union of a sequence of $l$ unit squares of $ \Ze^2$ such that any two consecutive squares of the sequence share a vertical (resp., horizontal) side of length $1$ in common. A {\it rectangle of size $r_1\times r_2$ ($r_1,r_2\in\Ze_+$) of $ \Ze^2$} is a union of a sequence of $r_1$ (resp., $r_2$) vertical (resp., horizontal) bars of length $r_2$ (resp., $r_1$) of $ \Ze^2$ such that any two consecutive vertical (resp., horizontal) bars of the sequence share a vertical (resp., horizontal) side of length $r_2$ (resp., $r_1$) in common. Finally, let $N=m(m+\epsilon)+k\geq 4$ be the unique decomposition of the integer $N\geq 4$ such that $k, m$, and $\epsilon$ are integers satisfying $\epsilon\in\{0,1\}$ and $0\leq k<m+\epsilon$. (Clearly, $m\geq 2$.) Then a {\it basic polyomino spanned by $N$ integer points of $ \Ze^2$} is a union of a rectangle of size $(m-1)\times (m-1+\epsilon)$ and a vertical bar of length $k-1$ such that the top sides of the rectangle and the bar sit on the same horizontal line sharing only one point in common namely, a lattice point of $\Ze^2$ (Figure~\ref{Figure5.pdf}). Here, if $k=1$, then a vertical bar of length $0$ means a horizontal line segment of length $1$ and as such possesses $0$ unit square, but it has one side of length $1$. Finally, {\it the graph of a basic polyomino} has the same vertex set as the the basic polyomino, which is the set of the generating integer points of the basic polyomino and its edges are the sides of length $1$ of the basic polyomino. Now, recall the following observation from \cite{AlCe}.

\begin{claim}\label{Alonso-Cerf}
The number of edges of the graph of any basic polyomino spanned by $N\geq 4$ integer points of $ \Ze^2$ is equal to $c_{ \Ze^2}(N)=\lfloor 2N-2\sqrt{N}\rfloor$.
\end{claim}

\begin{figure}[ht]
\begin{center}
\includegraphics[width=0.6\textwidth]{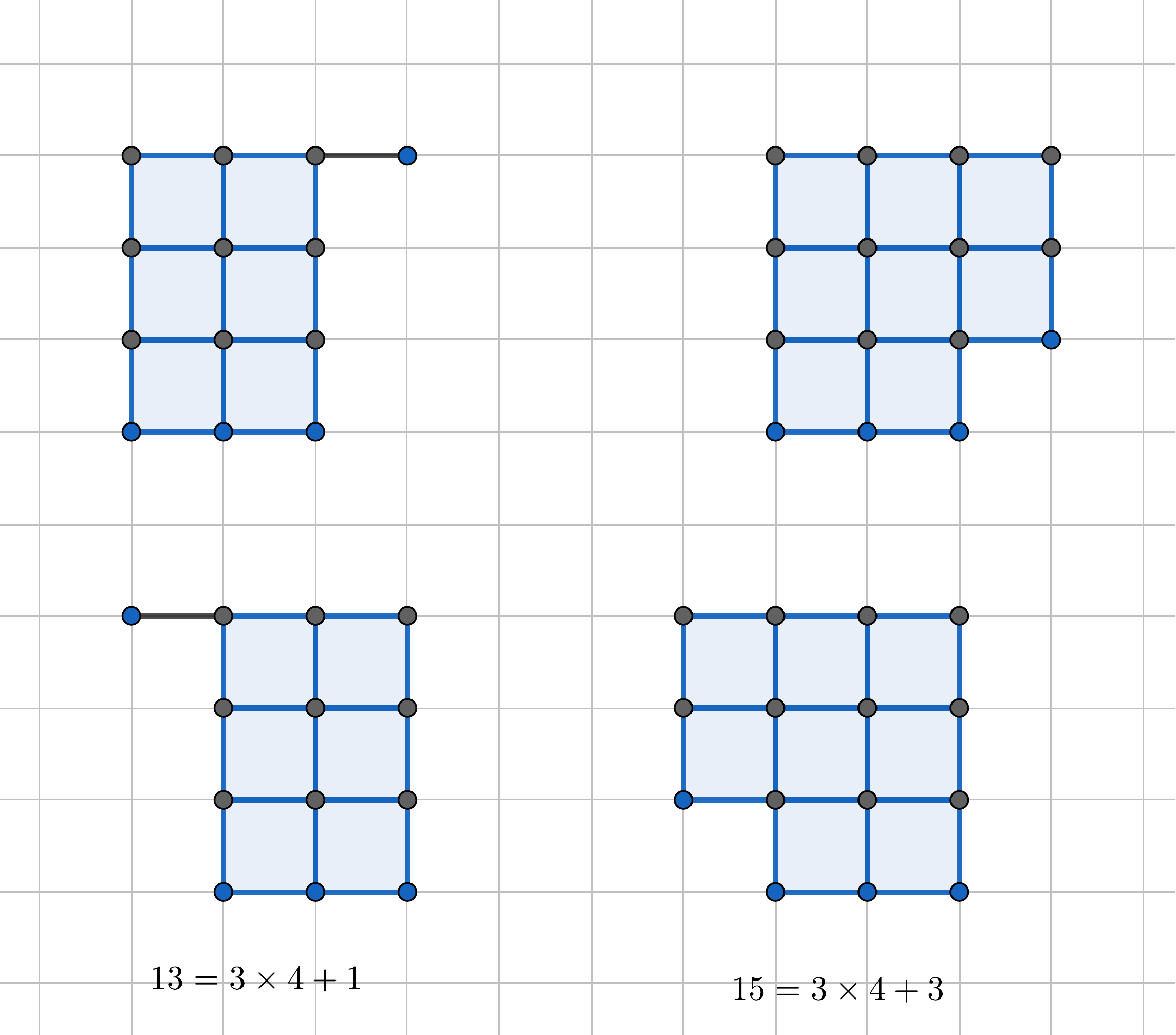}
\caption{Four basic polyominoes of $ \Ze^2$ with vertical bars of length 0 and of length 2 and their graphs.}
\label{Figure5.pdf} 
\end{center}
\end{figure}

By assumption $G_c(\mathcal {P})$ possesses two induced subgraphs $G_1$ and $G_2$ with only one vertex in common and with $n_1\geq 2$ and $n_2\geq 2$ vertices, respectively, where $n=n_1+n_2-1\geq 8$.  
Then either {\it (A)} $n_1\geq 4$ and $n_2\geq 4$, or {\it (B)} $n_1=3$, or {\it (C)} $n_2=3$, or {\it (D)} $n_1=2$, or {\it (E)} $n_2=2$.

{\it (A)}:  Let $G'_1$ (resp., $G'_2$) be the graph of a basic polyomino spanned by $n_1$ (resp., $n_2$) integer points of $\Ze^2$ such that $G'_1$ (resp., $G'_2$) lies in the first (resp., third) quadrant of the Cartesian coordinate system of $\Ee^2$ and their only vertex in common is the origin $\oo$ of $\Ee^2$. Also, we assume that the vertical bar of the basic polyomino whose graph is $G'_1$ (resp., $G'_2$) lies to the right (resp., left) of its underlying rectangle. Using Theorem~\ref{2D-theorem} and Claim~\ref{Alonso-Cerf} it follows that the number of edges of $G_1$ and $G'_1$ (resp., $G_2$ and $G'_2$) is equal to $\lfloor 2n_1-2\sqrt{n_1}\rfloor$ (resp., $\lfloor 2n_2-2\sqrt{n_2}\rfloor$) moreover,
\begin{equation}\label{equation28}
\lfloor 2n-2\sqrt{n}\rfloor=\lfloor 2n_1-2\sqrt{n_1}\rfloor+\lfloor 2n_2-2\sqrt{n_2}\rfloor
\end{equation}
It is easy to show (see Figures~\ref{Figure6.pdf},~\ref{Figure7.pdf} and~\ref{Figure8.pdf}) that one can always delete some vertices of the graph $G'_1$ or $G'_2$ and add them back by choosing new positions for them such that one obtains the contact graph of an LS packing of $n$ unit diameter disks having $\lfloor 2n-2\sqrt{n}\rfloor+1$ edges guaranteed by (\ref{equation28}). This contradicts Theorem~\ref{2D-theorem}, proving that {\it (A)} cannot hold.

\begin{figure}[ht]
\begin{center}
\includegraphics[width=0.7\textwidth]{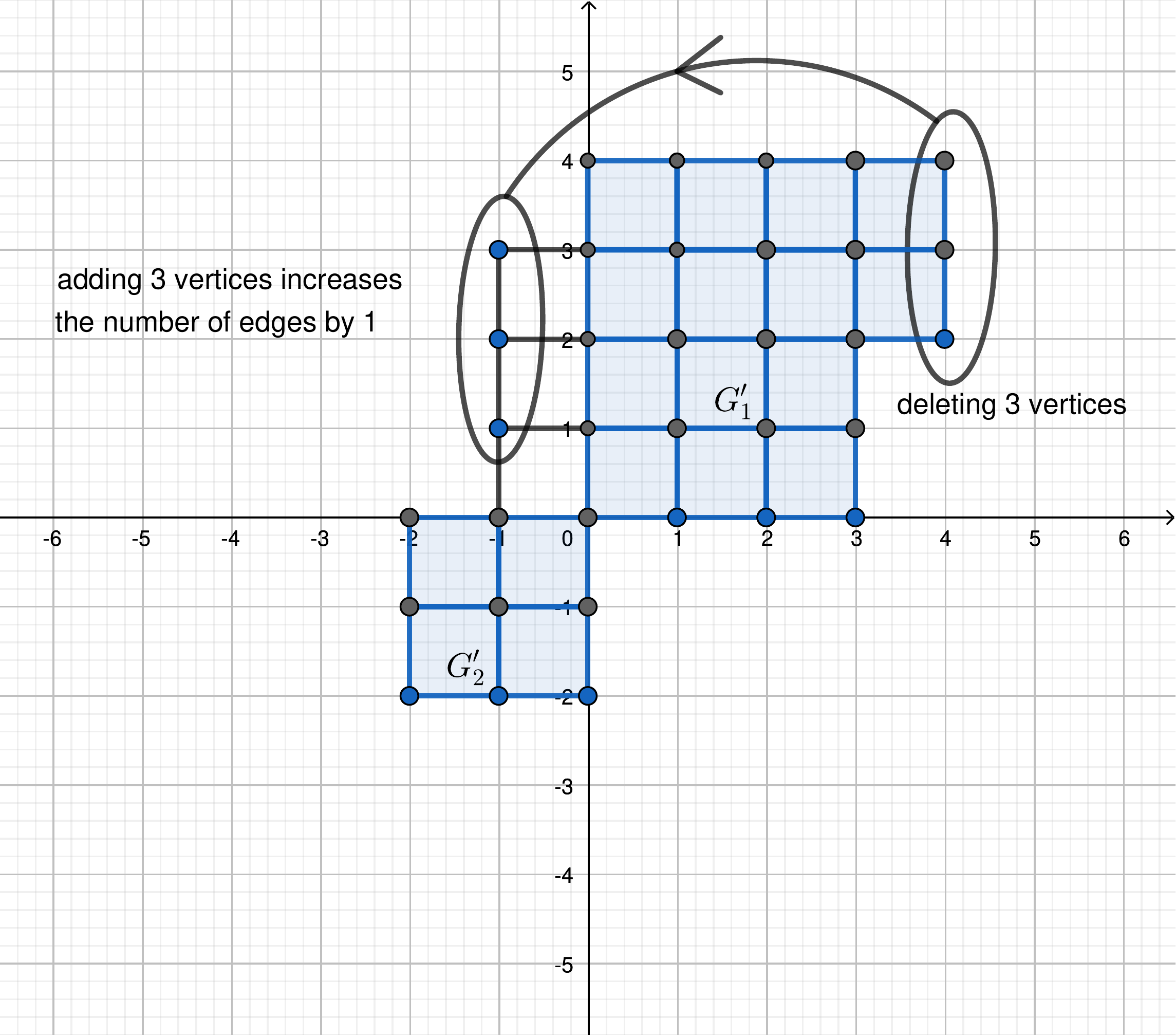}
\caption{The subcase of {\it (A)} when the basic polyomino of the graph $G'_1$ possesses a vertical bar.}
\label{Figure6.pdf} 
\end{center}
\end{figure}

\begin{figure}[ht]
\begin{center}
\includegraphics[width=0.7\textwidth]{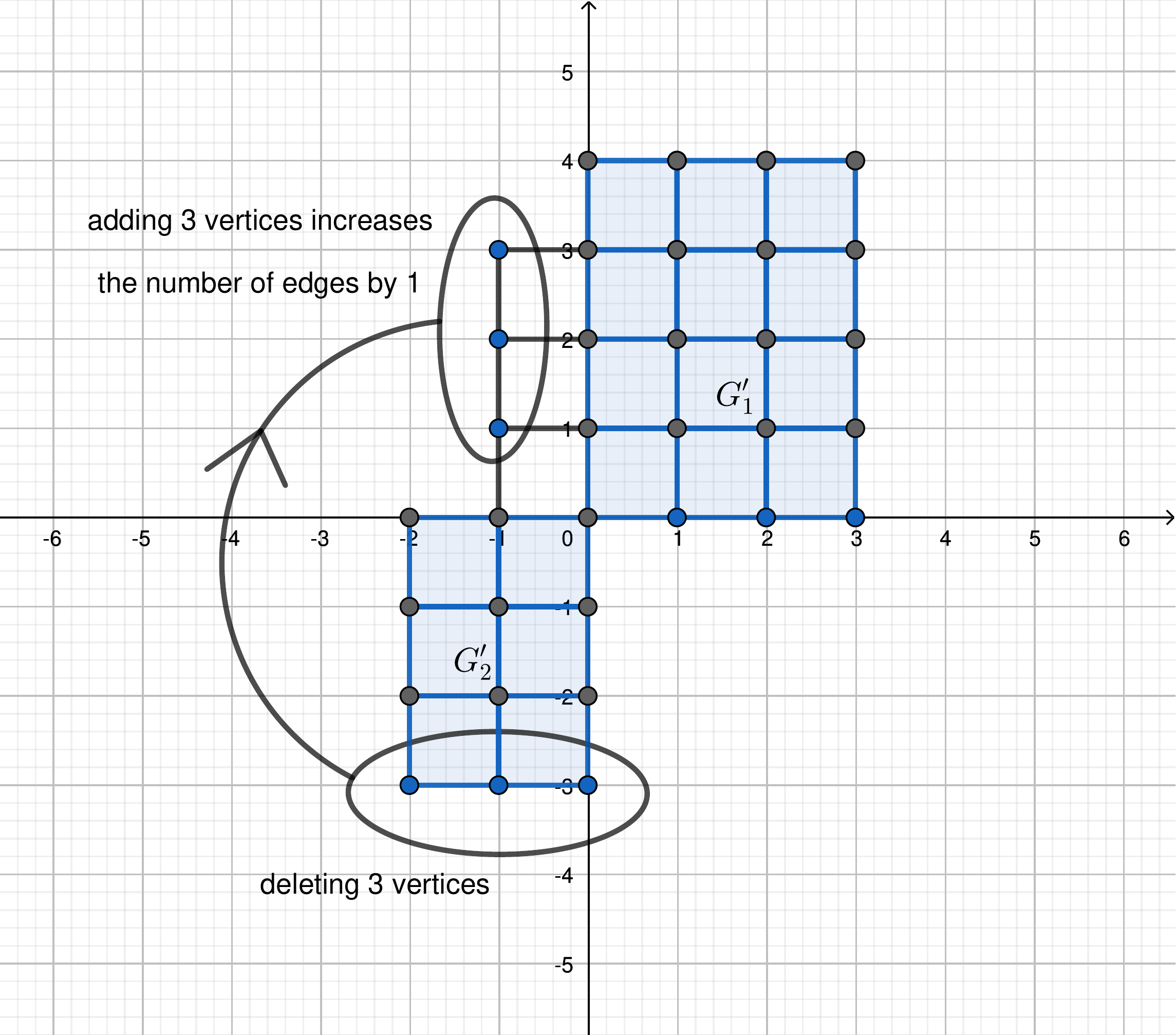}
\caption{The subcase of {\it (A)} when the basic polyominoes of the graphs $G'_1$ and $G'_2$ are rectangles. Here the vertical size of the rectangle of $G'_1$ is strictly larger than the horizontal size of the rectangle of $G'_2$.} 
\label{Figure7.pdf} 
\end{center}
\end{figure}

\begin{figure}[ht]
\begin{center}
\includegraphics[width=0.7\textwidth]{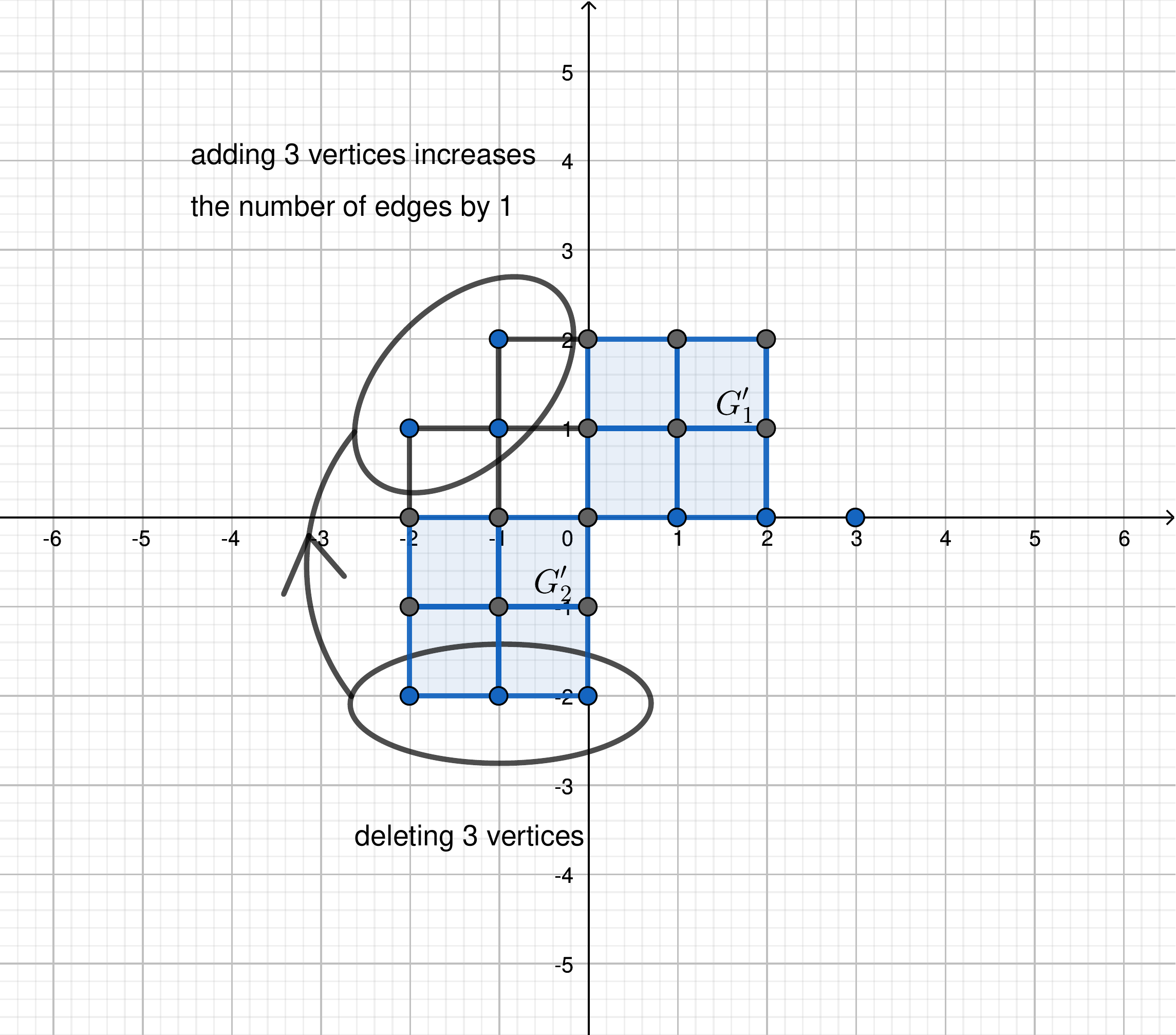}
\caption{The subcase of {\it (A)} when the basic polyominoes of the graphs $G'_1$ and $G'_2$ are squares of size $\geq 2$.} 
\label{Figure8.pdf} 
\end{center}
\end{figure}

{\it (B) - (C)}: Deleting the two vertices of $G_1$ (resp., $G_2$) which are not in $G_2$ (resp., $G_1$) from $G_c(\mathcal {P})$ together with the two edges adjacent to them yields the contact graph of an LS-packing of $n-2$ unit diameter disks having maximum contact number. Thus, by induction Theorem~\ref{2D-crystallization} applies to this contact graph on $n-2$ vertices and so, it follows that one can add two vertices to it by choosing proper positions for them such that one obtains a contact graph having $\lfloor 2n-2\sqrt{n}\rfloor+1$ edges. This contradicts Theorem~\ref{2D-theorem}, showing that {\it (B)} (resp., {\it (C)}) cannot hold either.

{\it (D) - (E)}: Deleting the degree one vertex of $G_1$ (resp., $G_2$) which is not in $G_2$ (resp., $G_1$) from $G_c(\mathcal {P})$ together with the edge adjacent to it yields the contact graph of an LS-packing of $n-1$ unit diameter disks having maximum contact number. Thus, by induction Theorem~\ref{2D-crystallization} applies to this contact graph on $n-1$ vertices and so, one obtains that it must be the type described in Case (i) of Theorem~\ref{2D-crystallization} (otherwise one could increase the number of edges of $G_c(\mathcal {P})$ by one, a contradiction). It follows that $G_c(\mathcal {P})$ possesses the properties listed under Case (iii) of Theorem~\ref{2D-crystallization}.

\vskip2.0cm

{\Large{\bf Acknowledgements}\par}
\bigskip

\noindent The author would like to thank the anonymous referee for careful reading and valuable comments.

\bigskip


\noindent K\'aroly Bezdek \\
\small{Department of Mathematics and Statistics, University of Calgary, Canada}\\
\small{Department of Mathematics, University of Pannonia, Veszpr\'em, Hungary\\
\small{E-mail: \texttt{bezdek@math.ucalgary.ca}}


\small
\begin{thebibliography}{GGM}

\bibitem{AlCe}
L. Alonso and R. Cerf, The three-dimensional polyominoes of minimal area, {\it Electron. J. Combin.} {\bf 3/1} (1996) Research Paper 27, approx. 39 pp.

\bibitem{Ba}
E. Baranovskii, On packing $n$-dimensional Euclidean space by equal spheres, {\it Izv. Vyssh. Uchebn. Zaved. Mat.} {\bf 39/2} (1964), 14--24.


\bibitem{Be02}
K. Bezdek, On the maximum number of touching pairs in a finite packing of translates of a convex body, {\it J. Combin. Theory Ser. A } {\bf 98/1} (2002), 192--200.

\bibitem{Bez02}
K. Bezdek, Improving Rogers' upper bound for the density of unit ball packings via estimating the surface area of Voronoi cells from below in Euclidean d-space for all $d\geq 8$, {\it Discrete Comput. Geom.} {\bf 28/1} (2002), 75--106.


\bibitem{BeRe}
K. Bezdek and S. Reid, Contact graphs of unit sphere packings revisited, {\it J. Geom.} {\bf 104 /1} (2013), 57--83.

\bibitem{BeSzSz}
K. Bezdek, B. Szalkai, and I. Szalkai, On contact numbers of totally separable unit sphere packings, {\it Discrete Math.} {\bf 339/2} (2016), 668--676. 


\bibitem{BeKh}
K. Bezdek and M. A. Khan, Contact numbers for sphere packings, in New Trends in Intuitive Geometry, 25--47, {\it Bolyai Soc. Math. Stud.} {\bf 27}, J\'anos Bolyai Math. Soc., Budapest, 2018.

\bibitem{BeLa}
K. Bezdek and Zs. L\'angi, Minimizing the mean projections of finite $\rho$--separable packings, {\it Monatsh. Math.} {\bf 188/4} (2019), 611--620.


\bibitem{BoDoMu}
P. Boyvalenkov, S. Dodunekov, and O. Musin, A survey on the kissing numbers, {\it Serdica Math. J.} {\bf 38/4} (2012), 507--522.


\bibitem{Bo}
K. B\"or\"oczky, Packing of equal spheres in spaces of constant curvature, {\it Acta Math. Acad. Sci. Hungar.} {\bf 32/3?4} (1978), 243--261.


\bibitem{DaHa}
H. Davenport and G. Haj\'os, Problem 35 (in Hungarian), {\it Mat. Lapok} {\bf 2} (1951), 68.

\bibitem{FeFe}
G. Fejes T\'oth and L. Fejes T\'oth, On totally separable domains, {\it Acta Math. Acad. Sci. Hungar.} {\bf 24} (1973), 229--232.

\bibitem{Ha}
T. C. Hales, Dense Sphere Packing - A Blueprint for Formal Proofs, {\it London
Mathematical Society Lecture Note Series} {\bf 400}, Cambridge University Press, Cambridge, 2012.

\bibitem{HaHa}
F. Harary and H. Harborth, Extremal animals, {\it J. Comb. Inf. Syst. Sci.} {\bf 1/1} (1976), 1--8.

\bibitem{Har}
H. Harborth, L\"osung zu Problem 664A, {\it Elem. Math.} {\bf 29} (1974), 14--15.

\bibitem{HeRa}
R. C. Heitmann and C. Radin, The ground state for sticky disks, {\it J. Statist. Phys.} {\bf 22/3} (1980), 281--287.

\bibitem{JeJoPe}
M. Jenssen, F. Joos, and W. Perkins, On the hard sphere model and sphere packings in high dimensions, {\it Forum Math. Sigma} {\bf 7} (2019), Paper No. e1, 19 pp.

\bibitem{Ke}
G. Kert\'esz, On totally separable packings of equal balls, {\it Acta Math. Hungar.} {\bf 51/3-4} (1988), 363--364.

\bibitem{Ku}
W. Kuperberg, Optimal arrangements in packing congruent balls in a spherical container, {\it Discrete Comput. Geom.} {\bf 37/2} (2007), 205--212.


\bibitem{Ne}
E. J. Neves, A discrete variational problem related to Ising droplets at low temperatures, {\it J. Statist. Phys.} {\bf 80/1-2} (1995), 103--123.

\bibitem{Oss}
R. Osserman, The isoperimetric inequality, {\it Bull. Amer. Math. Soc.} {\bf 84/6} (1978), 1182--1238.


\bibitem{Ran}
R. A. Rankin, The closest packing of spherical caps in n dimensions, {\it Proc. Glasgow Math. Assoc.} {\bf 2} (1955), 139--144.

\bibitem{Ro}
C. A. Rogers, The packing of equal spheres, {\it Proc. London Math. Soc.} {\bf 3/8} (1958), 609--620.


\bibitem{Rog}
C. A. Rogers, Packing and Covering, {\it Cambridge Univ. Press}, Cambridge, 1964.

\end{thebibliography}
\end{document}